\documentclass[11pt,reqno,a4paper]{amsart}
\usepackage{amsmath,amsthm,amscd,amssymb,amsfonts, amsbsy,upref,stmaryrd}
\usepackage{latexsym}
\usepackage{mathrsfs}
\usepackage{txfonts}
\usepackage{graphicx}%
\usepackage{fancyhdr}
\usepackage{pgf}
\usepackage{color}
\usepackage{enumerate}
\usepackage{bbm}
\usepackage{exscale}

\theoremstyle{plain} 
\numberwithin{equation}{section}
\newtheorem{theorem}{Theorem}[section]
\newtheorem{corollary}[theorem]{Corollary}

\newtheorem{lemma}[theorem]{Lemma}

\newtheorem{proposition}[theorem]{Proposition}
\theoremstyle{definition}

\newtheorem{remark}[theorem]{Remark}

\theoremstyle{definition}
\newtheorem{definition}[theorem]{Definition}

\newcommand{\appendicesB}{\par
\if@chapter@pp
\setcounter{chapter}{0}%
\setcounter{section}{0}%
\gdef\@chapapp{\appendixname}%
\gdef\thechapter{B\c@chapter}
\else
\setcounter{section}{0}%
\setcounter{subsection}{0}%
\gdef\thesection{B}
\fi
}
\newcommand{\appendicesA}{\par
\if@chapter@pp
\setcounter{chapter}{0}%
\setcounter{section}{0}%
\gdef\@chapapp{\appendixname}%
\gdef\thechapter{A\c@chapter}
\else
\setcounter{section}{0}%
\setcounter{subsection}{0}%
\gdef\thesection{A}
\fi
}
\def\@makechapterhead#1{      \null
     \begin{center} 
       APPENDIX \thechapter\\
       #1
      \end{center}
     \nobreak
}

\begin{document}

\title[Comparison of T1 conditions for multiparameter operators]{Comparison of T1 conditions for multiparameter operators}

\author{Ana Grau de la Herr\'an}
\address{DEPARTMENT OF MATHEMATICS AND STATISTICS, UNIVERSITY OF HELSINKI, FINLAND}
\email{ana.grau@helsinki.fi}

\begin{abstract} Journ\'e \cite{J} established the classical multi-parameter singular integral theory whose formulation was written in the language of vector-valued Calder\'on-Zygmund theory. More recently, Pott and Villarroya \cite{PV} formulated a new type of $T1$ theorem for product spaces where the vector-valued formulations were replaced by several mixed type conditions. Later on, Martikainen \cite{M} redefined the biparameter operators inspired in the work of Pott and Villarroya. Here we intend to show that for $L^2$ bounded $T$, the classes are equals although perhaps not in general. 
\end{abstract}

\maketitle

\tableofcontents

\section{Introduction \label{s1}}

Journ\'e proved in \cite{J} the $T(1)$ Theorem for Calder\'on-Zygmund operators on product spaces. In that paper, Journ\'e was able to formulate the statement of the theorem in a way that a priori resembles the classical one by using vector valued Calder\'on-Zygmund theory formulation. Once we analyse more closely this formulation,  a priori boundedness of some components of the operator is required, which differs from the classical setting. This variance comes from trying to overcome some challenges that are not encountered in the classical case as, for example, that the singularities of multiparameter operators lie not only at the origin (as is the case of standard Calder\'on-Zygmund kernels), but they spread over larger subspaces. Pott and Villarroya \cite{PV} modified the formulation so that no a priori boundedness is assumed in the operator.

The relationship between these two classes of operators defined from the different formulations was unclear. In this paper we prove that for $L^2$ bounded operators the two sets of conditions actually define the same class of operators.

The main result of the paper reads as follows:

\begin{theorem}

Let $T:\mathcal{C}_0^{\infty}(\mathbb{R}^n)\otimes\mathcal{C}_0^{\infty}(\mathbb{R}^m)\rightarrow\left[\mathcal{C}_0^{\infty}(\mathbb{R}^n)\otimes\mathcal{C}_0^{\infty}(\mathbb{R}^m)\right]'$ be a continuous linear mapping ($n+m=d$) that has the kernel representation
\begin{equation*}
Tf(x)=\int_{\mathbb{R}^d}K(x,y)f(y)dy.
\end{equation*}

If $T$ can be extended to a bounded operator $T:L^2\rightarrow L^2$ then it satifies the Journ\'e type conditions, i.e, $T$ is a bi-parameter $\delta$-SIO as defined in Definition \ref{deltasio} satisfying the weak boundedness property \eqref{classWBP} if and only if it satisfies the Pott-Villarroya type conditions, i.e., $T$ is an operator defined as in \eqref{defmixoperator} whose kernel satisfies \eqref{size}--\eqref{adapted1} and, additionally,  $T$ also satisfies \eqref{mixedwbp}--\eqref{diagbmo4} .
\end{theorem} 

One of the reasons that leaded to compare both formulations is that Journ\'e proved that for $L^2$ bounded operators $T$ that satisfy the Journ\'e type conditions imply that $T1, \ T^*1\in BMO$, while this was previously not known for the Pott-Villarroya type conditions. Thus, for the new product $T1$ theorems, the \newline $T1, \ T^*1\in BMO$ conditions were just sufficient, but they were not known to be necessary. Moreover, with the equivalence of the Journ\'e type and Pott-Villarroya type conditions, we indirectly find that $T1, \ T^*1\in BMO$ also for the $L^2$ bounded operators $T$ that satisfy the Pott-Villarroya type conditions.

We want to stress out that even when the two sets of conditions are now found to be equivalent, the new Pott-Villarroya type of conditions are still useful, since it may be easier to verify them in concrete cases than the vector-valued Journ\'e type of conditions.

The layout of the paper is as follows. We are going to state the classical result and conditions as in \cite{J} in Section \ref{s2} while we will introduce the new mixed type conditions as they were defined in  \cite{M} in Section \ref{s3}. Then we will proceed to prove the relation of such conditions in Sections \ref{s4} and \ref{s5}.

\textbf{Acknowledgements}.- We would like to thank prof. Tuomas Hyt\"onen for suggesting this problem as well as multiple useful conversations that granted  important insight for the development of the paper. The author was supported by the European Union through the ERC Starting Grant "Analytic-probabilistic methods for borderline singular integrals". 
\section{Classical formulation \label{s2}}

In this section we are going to introduce the classical formulation as stated in Journ\'e's original paper.

Let $\Omega=\mathbb{R}^d\times\mathbb{R}^d\setminus\Delta$, where $\Delta=\{(x,y), \ x=y\}$ and let $\delta\in(0,1)$.

\begin{definition}\label{bdeltastandard}
Let $K$ be a continuous function defined on $\Omega$ and taking its values in a Banach space $B$. The function $K$ is a \textbf{ $B-\delta$-standard kernel} if the following are satisfied, for some constant $C>0$.

For all $(x,y)\in\Omega$,
\begin{equation}\label{classicalsize}
|K(x,y)|_B\leq\frac{C}{|x-y|^d}.
\end{equation}

For all $(x,y)\in\Omega$, and $x'\in\mathbb{R}^d$ such that $|x-x'|<\frac{|x-y|}{2}$,

\begin{equation}\label{classcancel1}
|K(x,y)-K(x',y)|_B\leq C\frac{|x-x'|^{\delta}}{|x-y|^{d+\delta}}
\end{equation}
and
\begin{equation}\label{classcancel2}
|K(y,x)-K(y,x')|_B\leq C\frac{|x-x'|^{\delta}}{|x-y|^{d+\delta}}.
\end{equation}
\end{definition}

The smallest constant C for which \eqref{classicalsize}, \eqref{classcancel1} and \eqref{classcancel2} hold is denoted by $|K|_{\delta, B}$. If the Banach space is the complex plane $\mathbb{C}$ we will omit the subscript $B$ for simplicity.

\begin{definition}
Let $T:\mathcal{C}_0^{\infty}(\mathbb{R}^d)\rightarrow\left[\mathcal{C}^{\infty}_0(\mathbb{R}^d)\right]'$ be a continuous linear mapping. $T$ is a singular integral operator (SIO) if, for some, $\delta\in(0,1)$, there exists a $\mathbb{C}$-$\delta$-standard kernel $K$ such that for all functions $f,g\in\mathcal{C}^{\infty}_0(\mathbb{R}^d)$ having disjoints supports
\begin{equation*}
<g,Tf>=\iint g(x)K(x,y)f(y)dydx.
\end{equation*}

We shall say that $T$ is a $\delta$-\textbf{SIO}.

\end{definition}

\begin{definition}
Let $T$ be a $\delta$-SIO and $K$ its kernel. We say that $T$ is a \textbf{$\delta$- Calder\'on-Zygmund operator} ($\delta$-CZO) if it extends boundedly from $L^2$ to itself. We also define the norm $\Vert\cdot\Vert_{\delta CZ}$ by
\begin{equation}\label{CZnorm}
\Vert T\Vert_{\delta CZ}=\Vert T\Vert_{2\rightarrow 2}+\vert K\vert_{\delta}
\end{equation}
\end{definition}

Note that the defined norm makes the set of $\delta$-CZO's a Banach space which we denote by $\delta CZ$.

\begin{remark}
To avoid excessive complication on notation we shall write $\vert T\vert_{\delta}=\vert K\vert_{\delta}$ where by $K$ we mean the kernel of $T$.
\end{remark}

\begin{definition}\label{deltasio}\cite{J}
Let $T:\mathcal{C}^{\infty}_0(\mathbb{R}^n)\otimes\mathcal{C}_0^{\infty}(\mathbb{R}^m)\rightarrow\left[\mathcal{C}_0^{\infty}(\mathbb{R}^n)\otimes\mathcal{C}_0^{\infty}(\mathbb{R}^m)\right]'$ be a continuous linear mapping. It is a \textbf{bi-parameter $\delta$-SIO on $\mathbb{R}^n\times\mathbb{R}^m$} if there exists a pair $(K_1,K_2)$ of $\delta CZ$-$\delta$-standard kernels so that, for all $ f_1,g_1\in\mathcal{C}_0^{\infty}(\mathbb{R}^n)$ and $f_2,g_2\in\mathcal{C}_0^{\infty}(\mathbb{R}^m)$, with $supp f_i\cap supp g_i=\emptyset$ ($i=1,2$),  
\begin{equation}\label{defK1}
\langle g_1\otimes g_2, Tf_1\otimes f_2\rangle=\iint g_1(x_1)\langle g_2,K_1(x_1,y_1)f_2\rangle f_1(y_1)dx_1dy_1,
\end{equation}
\begin{equation}\label{defK2}
\langle g_1\otimes g_2, Tf_1\otimes f_2\rangle=\iint g_2(x_2)\langle g_1, K_2(x_2,y_2)f_1\rangle f_2(y_2)dx_2dy_2.
\end{equation}

Let $\tilde{T}$ be defined by 
\begin{equation}
\langle g\otimes k, \tilde{T}f\otimes h\rangle=\langle f\otimes k, Tg\otimes h\rangle.
\end{equation}

It is readly seen that $\tilde{T}$ is a bi-parameter $\delta$-SIO if $T$ is. Its kernels $\tilde{K}_1$ and $\tilde{K}_2$ will be given by $\tilde{K}_1(x,y)=K_1(y,x)$ and $\tilde{K}_2(x,y)=\left[K_2(x,y)\right]^*.$
\end{definition}

Furthermore, let us introduce some notation for simplicity purposes. We define the operator $\langle g_1,T^1f_1\rangle:\mathcal{C}_0^{\infty}\rightarrow\left[\mathcal{C}_0^{\infty}(\mathbb{R})\right]'$ by
\begin{equation}
\langle g_2,\langle g_1, T^1f_1\rangle f_2\rangle=\langle g_1\otimes g_2, Tf_1\otimes f_2\rangle.
\end{equation}

It is easy to check that $\langle g_1, T^1f_1\rangle$ is a $\delta$-SIO on $\mathbb{R}$ with kernel
\begin{equation}\label{intermediatekernel}
K^1_{f_1,g_1}(x_2,y_2):=\langle g_1,T^1f_1\rangle(x_2,y_2)=\langle g_1,K_2(x_2,y_2)f_1\rangle.
\end{equation} 

One defines $K^2_{f_2,g_2}:=\langle g_2,T^2f_2\rangle$ in a similar manner.

\begin{definition}\label{classWBP}
Let $T$ be a bi-paramenter $\delta-SIO$ on $\mathbb{R}^d\times\mathbb{R}^d$. We say it has the bi-parameter \textbf{weak boundedness property} (WBP) in the classical sense if for any bounded subset $\mathcal{B}$ of $\mathcal{C}_0^{\infty}(\mathbb{R}^d)$ there exists a positive constant C (depending in the bounded subset) such that for any pair $(\eta,\xi)\in\mathcal{B}\times\mathcal{B}$, any $x_i\in\mathbb{R}^d$, $t>0$ and $i\in\{1,2\}$,

\begin{equation}\label{eqclassWBP}
\Vert\langle \eta_t^{x_i},T^i\xi^{x_i}_t\rangle\Vert_{\delta CZ}\leq C_{\mathcal{B}} t^{-d_i}
\end{equation}
\noindent where $\eta_t^{x_i}(z_i)=\frac{1}{t^{d_i}}\eta\left(\frac{z_i-x_i}{t}\right)$ ($\xi_t^{x_i}$ defined similarly),  $d_1=n$, $d_2=m$ and $T^i$ defined as above.
\end{definition}


\section{Mixed type conditions formulation \label{s3}}

In this section we are going to introduce the mixed type conditions formulation introduced by Pott and Villarroya \cite{PV} as reformulated by Martikainen \cite{M}.

\begin{definition}
We say that a function $u_V$ is \textbf{V-adapted with zero mean} if it satisfies $supp(u_V)\subset V, \ |u_V|\leq 1$ and $\int u_V=0$.
\end{definition}
\begin{definition}
Let $T:\mathcal{C}_0^{\infty}(\mathbb{R}^n)\otimes\mathcal{C}_0^{\infty}(\mathbb{R}^m)\rightarrow\left[\mathcal{C}_0^{\infty}(\mathbb{R}^n)\otimes\mathcal{C}_0^{\infty}(\mathbb{R}^m)\right]'$ be a continuous linear mapping $(n+m=d)$. Let $f=f_1\otimes f_2$ and $g=g_1\otimes g_2$ with $f_1,g_1:\mathbb{R}^n\rightarrow\mathbb{C}$, $f_2,g_2:\mathbb{R}^m\rightarrow\mathbb{C}$ satisfying $supp f_i\cap supp g_i=\emptyset$ for $i\in\{1,2\}$. We denote $f=f_1\otimes f_2$ (meaning $f(x)=f_1(x_1)\cdot f_2(x_2)$ for $x=(x_1,x_2)$) and $g=g_1\otimes g_2$.

We say that $T$ has a \textbf{Calder\'on-Zygmund structure} if it has the kernel representation
\begin{equation}\label{defmixoperator}
<Tf,g>=\int_{\mathbb{R}^d}\int_{\mathbb{R}^d}K(x,y)f(y)g(x)dxdy
\end{equation}
\noindent where the kernel $K:(\mathbb{R}^d\times\mathbb{R}^d)\setminus\{(x,y)\in\mathbb{R}^d\times\mathbb{R}^d \ : \ x_1=y_1 \ or \ x_2=y_2\}\rightarrow\mathbb{C}$ is assumed to satisfy the following conditions:
\begin{itemize}
\item \textbf{Size condition} \begin{equation}\label{size}
|K(x,y)|\leq C\frac{1}{|x_1-y_1|^n}\frac{1}{|x_2-y_2|^m}
\end{equation}

Aditionally, when $|x_i-x_i'|\leq|x_i-y_i|/2$ and $|y_i-y_i'|\leq |x_i-y_i|/2$ $i=1,2$

\item \textbf{H\"older condition}\begin{equation}\label{holder1}
|K(x,y)-K(x,(y_1,y_2'))-K(x,(y_1',y_2))+K(x,y')|\leq C\frac{|y_1-y_1'|^{\delta}}{|x_1-y_1|^{n+\delta}}\frac{|y_2-y_2'|^{\delta}}{|x_2-y_2|^{m+\delta}}
\end{equation}

\begin{equation}\label{holder2}
|K(x,y)-K((x_1,x_2'),y)-K((x_1',x_2),y)+K(x',y)|\leq C\frac{|x_1-x_1'|^{\delta}}{|x_1-y_1|^{n+\delta}}\frac{|x_2-x_2'|^{\delta}}{|x_2-y_2|^{m+\delta}}
\end{equation}

\begin{align}\label{holder3}
|K(x,y)-K((x_1,x_2')y)-K(x,(y_1',y_2))+ & K((x_1,x_2'),(y_1',y_2))|\\
&\leq C\frac{|y_1-y_1'|^{\delta}}{|x_1-y_1|^{n+\delta}}\frac{|x_2-x_2'|^{\delta}}{|x_2-y_2|^{m+\delta}}\notag
\end{align}

\begin{align}\label{holder4}
|K(x,y)-K(x,(y_1,y_2'))-K((x_1',x_2),y)+ & K((x_1',x_2),(y_1,y_2'))|\\
&\leq C\frac{|x_1-x_1'|^{\delta}}{|x_1-y_1|^{n+\delta}}\frac{|y_2-y_2'|^{\delta}}{|x_2-y_2|^{m+\delta}}\notag
\end{align}

\item \textbf{Mixed H\"older and size conditions} \begin{equation}\label{mixed1}
|K(x,y)-K((x_1',x_2),y)|\leq C\frac{|x_1-x_1'|^{\delta}}{|x_1-y_1|^{n+\delta}}\frac{1}{|x_2-y_2|^{m}}
\end{equation}

\begin{equation}\label{mixed2}
|K(x,y)-K(x,(y_1',y_2))|\leq C \frac{|y_1-y_1'|^{\delta}}{|x_1-y_1|^{n+\delta}}\frac{1}{|x_2-y_2|^{m}}
\end{equation}

\begin{equation}\label{mixed3}
|K(x,y)-K((x_1,x_2'),y)|\leq C\frac{1}{|x_1-y_1|^{n}}\frac{|x_2-x_2'|^{\delta}}{|x_2-y_2|^{m+\delta}}
\end{equation}

\begin{equation}\label{mixed4}
|K(x,y)-K(x,(y_1,y_2'))|\leq C \frac{1}{|x_1-y_1|^{n}}\frac{|y_2-y_2'|^{\delta}}{|x_2-y_2|^{m+\delta}}
\end{equation}

\item \textbf{Separated H\"older and size conditions}
\begin{equation}\label{sep1}
|K^j_{f_j,g_j}(x_i,y_i)|\leq C(f_j,g_j)\frac{1}{|x_i-y_i|^{d_i}}
\end{equation}

\begin{equation}\label{sep2}
|K^j_{f_j,g_j}(x_i,y_i)-K^j_{f_j,g_j}(x_i',y_i)|\leq C(f_j,g_j)\frac{|x_i-x_i'|^{\delta}}{|x_i-y_i|^{d_i+\delta}}
\end{equation}

\begin{equation}\label{sep3}
|K^j_{f_j,g_j}(x_i,y_i)-K^j_{f_j,g_j}(x_i,y_i')|\leq C(f_j,g_j)\frac{|y_i-y_i'|^{\delta}}{|x_i-y_i|^{d_i+\delta}}
\end{equation}

\noindent where $i=1,2$ , $j=1,2$, $d_1=n$, $d_2=m$. Morever for all cubes $V\in\mathbb{R}^{d_j}$
\begin{equation}\label{adapted1}
C(\chi_V,\chi_V)+C(\chi_V,u_V)+C(u_V,\chi_V)\leq C|V|
\end{equation}
\noindent whenever $u_V$ is a V-adapted with zero mean.
\end{itemize}

Here $K^j_{f_j,g_j}$ is defined as in \eqref{intermediatekernel}.
\end{definition}

\begin{lemma}\label{genadapted1}
Let $T$ be an operator defined as in \eqref{defmixoperator} whose kernel satisfies conditions \eqref{sep1}--\eqref{adapted1} then for all cubes $V\in\mathbb{R}^{d_j}$
\begin{equation}
C(\chi_V,g_V)+C(g_V,\chi_V)\leq C \ max(1,\Vert g_V\Vert_{\infty}) \ |V|
\end{equation}
\noindent whenever $g_V\in L^{\infty}(V)$, $d_1=n, \ d_2=m$.
\end{lemma}
\begin{proof}

Let $g_V\in L^{\infty}(V)$ and rewrite it as follows

$$g_V=\left(g_V-\left(\fint_Vg_V\right)\chi_V\right)+\left(\fint_Vg_V\right)\chi_V=g_V^1+g_V^2$$

It is trivial to check that $\tfrac{1}{2\Vert g_V\Vert_{\infty}} g_V^1$ is V-adapted with zero mean and $g_V^2$ is a constant between 0 and $\Vert g_V\Vert_{\infty}$ multiplying the characteristic function restricted to V so by linearity and \eqref{adapted1}

\begin{align*}
C(\chi_V,g_V)&\leq\left(2\Vert g_V\Vert_{\infty} C(\chi_V,\tfrac{1}{2\Vert g_V\Vert_{\infty}}g_V^1)+\Vert g_V\Vert_{\infty} C(\chi_V,\chi_V)\right)\\
&\leq C \ max(1,\Vert g_V\Vert_{\infty}) \ |V|
\end{align*}

By symmetry we get the $C(g_V,\chi_V)\leq C \ max(1,\Vert g_V\Vert_{\infty}) \ |V|$.
\end{proof}
\begin{definition}
We say that $T$ satisfies the \textbf{weak boundedness property} in the mixed type sense if for every $Q\subset\mathbb{R}^n$ and $V\subset\mathbb{R}^m$

\begin{equation}\label{mixedwbp}
\langle T(\chi_Q\otimes\chi_V),\chi_Q\otimes\chi_V\rangle|\leq C|Q| \ |V|
\end{equation}
\end{definition}

To avoid confusion with the WBP in the classical sense defined in \eqref{classWBP} we are going to refer to \eqref{mixedwbp} as \textbf{mixed WBP}.

\begin{definition}
We say that $T$ satisfies \textbf{diagonal BMO conditions} if for every cube $Q\subset\mathbb{R}^n$ and $V\in\mathbb{R}^m$ and for every zero-mean functions $a_Q, b_V$ wich are $Q$ and $V$ adapted respectively:
\begin{equation}\label{diagbmo1}
|\langle T(a_Q\otimes\chi_V),\chi_Q\otimes\chi_V\rangle|\leq C|Q| \ |V|
\end{equation}

\begin{equation}\label{diagbmo2}
|\langle T(\chi_Q\otimes\chi_V),a_Q\otimes\chi_V\rangle|\leq C |Q| \ |V|
\end{equation}

\begin{equation}\label{diagbmo3}
|\langle T(\chi_Q\otimes b_V), \chi_Q\otimes\chi_V\rangle|\leq C|Q| \ |V|
\end{equation}

\begin{equation}\label{diagbmo4}
|\langle T(\chi_Q\otimes\chi_V),\chi_Q\otimes b_V\rangle|\leq C |Q| \ |V|
\end{equation}

\end{definition}
\section{Mixed type conditions imply classical conditions \label{s4}}

To prove that an operator $T$ that satisfies the mixed type conditions introduced in Section \ref{s3} is a bi-parameter $\delta$-SIO on $\mathbb{R}^d\times\mathbb{R}^d$ as defined in Section \ref{s2} we first need to find a pair of $\delta CZ$-$\delta$-standard kernels satisfying conditions \eqref{defK1} and \eqref{defK2}. Afterwards we are going to prove that if such $\delta$-SIO defines an $L^2$ bounded operator, it also satisfies the bi-parameter WBP \eqref{classWBP} in the classical sense. First of all we are going to recall the following version of the uniparametric $T(1)$ Theorem.





\begin{theorem}\label{1auxT1}\cite{Ho}
Let T be a $\delta$-SIO on $\mathbb{R}^d$ and K its kernel. If there exists a constant $A>0$ such that for every cube $V\subset\mathbb{R}^d$
\begin{equation}\label{L1}
\Vert T\chi_V\Vert_{L^1(V)}\leq A|V|
\end{equation}

and

\begin{equation}\label{L1adj}
\Vert T^*\chi_V\Vert_{L^1(V)}\leq A|V|
\end{equation}

Then $T$ is a bounded operator on $L^2$ such that $\Vert T\Vert_{2\rightarrow 2}\leq C_{\delta,d}\cdot (A+|K|_{\delta})$.

\end{theorem}

\begin{remark}
This version of the $T1$ theorem is not as well known as some others but follows, by a standard localization argument, from the classical versions.
\end{remark}

\begin{proposition}\label{mixedtoclassy}
Let $T$ be an operator defined as in \eqref{defmixoperator} whose kernel satisfies the conditions \eqref{size} -- \eqref{adapted1} , then the pair 
\begin{equation*}
\left(K_1(x_1,y_1),K_2(x_2,y_2)\right):= 
\left(K((x_1,\cdot),(y_1,\cdot)),K((\cdot,x_2),(\cdot,y_2))\right)
\end{equation*}
\noindent is a pair of $\delta CZ$-$\delta$- standard kernels satisfying conditions \eqref{defK1} and \eqref{defK2}.
\end{proposition}

\begin{proof}
That the pair of kernels satisfy conditions  \eqref{defK1} and \eqref{defK2} can be deduced from \eqref{defmixoperator} and Fubini so we are going to concentrate on proving that $K_1$ is a  $\delta CZ$-$\delta$- standard kernel and by the symmetry of the conditions we will also have that $K_2$ is $\delta CZ$-$\delta$- standard kernel.

Let's remind ourselves that for $K_1$ to be a $\delta CZ$-$\delta$- standard kernel, it needs to satisfy the size condition \eqref{classicalsize} and the cancellation conditions \eqref{classcancel1} and \eqref{classcancel2} where $|\cdot|_B=\Vert\cdot\Vert_{\delta CZ}$. This means that the kernel of $K_1$ (which is $K(x_1,y_1)(x_2,x_2)=$ $K((x_1,x_2),(y_1,y_2))=K(x,y)$  where the variables $x_1$ and $y_1$ are fixed) has to satisfy the aforementioned conditions with $|\cdot|_{B}$ being the absolute value and as an operator to be bounded in $L^2$ as operator. Let's do this step by step

\begin{enumerate}
\item We prove that $\Vert K_1(x_1,y_1)\Vert_{\delta CZ}\leq \frac{C}{|x_1-y_1|^n}.$

\vspace{3mm}

\begin{itemize}
\item It's immediate that $|K_1(x_1,y_1)|_{\delta}\leq\frac{C}{|x_1-y_1|^n}$ by \eqref{size}, \eqref{mixed3} and \eqref{mixed4}.

\vspace{3mm}

\item We prove that $\Vert K_1(x_1,y_1)\Vert_{2\rightarrow 2}\leq \frac{C}{|x_1-y_1|^n}.$

\vspace{2mm}

 The $L^2$ boundedness is going to be a consequence of applying Theorem \ref{1auxT1}, which means that by duality, we need to prove that 

\begin{equation*}
|\langle K_1(x_1,y_1)\chi_V,g_V\rangle|+|\langle g_V,K_1(x_1,y_1)\chi_V\rangle|\leq\frac{C}{|x_1-y_1|^n}|V|
\end{equation*}
\noindent for all $g_V\in L^{\infty}(V)$ such that $\Vert g_V\Vert_{\infty}\leq 1$.



Then by linearity, \eqref{sep1} and lemma \ref{genadapted1}

\begin{align*}|\langle K_1(x_1,y_1)\chi_V,g_V\rangle|+& |\langle g_V,K_1(x_1,y_1)\chi_V\rangle|\leq\\
& \leq \left( C (\chi_V,g_V)+C(g_V,\chi_V)\right)\frac{1}{|x_1-y_1|^n}\leq \frac{C}{|x_1-y_1|^n}|V|.\end{align*}
\end{itemize}

\item We prove that $\Vert K_1(x_1,y_1)-K_1(x_1',y_1)\Vert_{\delta CZ}\leq C \frac{|x_1-x_1'|^{\delta}}{|x_1-y_1|^{n+\delta}}.$

\vspace{3mm}

\begin{itemize}
\item $|K_1(x_1,y_1)-K_1(x_1',y_1)|_{\delta}\leq C \frac{|x_1-x_1'|^{\delta}}{|x_1-y_1|^{n+\delta}}$ by \eqref{mixed1}, \eqref{holder2} and \eqref{holder4}.

\vspace{3mm}

\item $\Vert K_1(x_1,y_1)-K_1(x_1',y_1)\Vert_{2\rightarrow 2}\leq C \frac{|x_1-x_1'|^{\delta}}{|x_1-y_1|^{n+\delta}}$ which is satisfied by reasoning as in the first case using \eqref{sep2} instead of \eqref{sep1}.

\end{itemize}
\vspace{3mm}

\item We prove that $\Vert K_1(x_1,y_1)-K_1(x_1,y_1')\Vert_{\delta CZ}\leq C \frac{|y_1-y_1'|^{\delta}}{|x_1-y_1|^{n+\delta}}.$

\vspace{3mm}

\begin{itemize}
\item $|K_1(x_1,y_1)-K_1(x_1,y_1')|_{\delta}\leq C \frac{|y_1-y_1'|^{\delta}}{|x_1-y_1|^{n+\delta}}$ by \eqref{mixed2}, \eqref{holder3} and \eqref{holder1}.

\vspace{3mm}

\item $\Vert K_1(x_1,y_1)-K_1(x_1,y_1')\Vert_{2\rightarrow 2}\leq C \frac{|y_1-y_1'|^{\delta}}{|x_1-y_1|^{n+\delta}}$ which is satisfied by reasoning as in the first case using \eqref{sep3} instead of \eqref{sep1}.

\end{itemize}
\end{enumerate}
\end{proof}

\begin{proposition}
Let T be an operator defined as in \eqref{defmixoperator} that can be extended to an $L^2$ to $L^2$ bounded bi-parameter operator and whose kernel satisfies \eqref{size}--\eqref{adapted1}. Then $T$ satisfies the bi-parameter WBP \eqref{classWBP} in the classical sense.
\end{proposition}
\begin{proof}

We are going to assume without loss of generality that $i=1$ since by symmetry of the conditions the other case is proved in the same manner. Let's fix a bounded subset $\mathcal{B}$ of $\mathcal{C}_0^{\infty}(\mathbb{R}^n)$. Then there exists a constant $C_{\mathcal{B}}$ such that $\Vert f\Vert_2\leq C_{\mathcal{B}}$ $\forall f\in\mathcal{B}$.

\vspace{2mm}

By definition, we need to prove that $\Vert\langle\eta_t^{x_1},T^1\xi_t^{x_1}\rangle\Vert_{\delta CZ}\leq\frac{C_{\mathcal{B}}}{t^n}$. Remember that in \eqref{intermediatekernel} we determined that

$$K^1_{\xi_t^{x_1},\eta_t^{x_1}}(x_2,y_2)=\langle\eta_t^{x_1},K_2(x_2,y_2)\xi_t^{x_1}\rangle=\langle\eta_t^{x_1}, T^1 \xi_t^{x_1}\rangle(x_2,y_2)$$

\vspace{2mm} 

First of all we are going to prove that $|K^1_{\xi_t^{x_1},\eta_t^{x_1}}|_{\delta}\leq\frac{\tilde{C}_{\mathcal{B}}}{t^n}$ by using the proof of Proposition \ref{mixedtoclassy} where we determined the $L^2$ boundedness of $K_1(x_1,y_1)$ and $K_2(x_2,y_2)$ as well as their H\"older versions.

\begin{itemize}

\item We prove that $|\langle \eta_t^{x_1}T^1\xi_t^{x_1}\rangle(x_2,y_2)|\leq \frac{C}{|x_2-y_2|^m}\frac{C_{\mathcal{B}}}{t^n}:$

\begin{align*}
|\langle \eta_t^{x_1},T^1\xi_t^{x_1}\rangle(x_2,y_2)|&=|\langle \eta_t^{x_1},K_2(x_2,y_2)\xi_t^{x_1}\rangle|\\
&\leq\Vert K_2(x_2,y_2)\Vert_{2\rightarrow 2} \ \Vert\xi_t^{x_1}\Vert_2 \ \Vert \eta_t^{x_1}\Vert_2\\
&\leq \frac{C}{|x_2-y_2|^m} \cdot \frac{C_{\mathcal{B}}}{t^n}
\end{align*}

\item Similarly we prove that $|\langle \eta_t^{x_1}T^1\xi_t^{x_1}\rangle(x_2,y_2) - \langle \eta_t^{x_1}T^1\xi_t^{x_1}\rangle(x_2',y_2)|$ 

\noindent $\leq C \frac{|x_2-x_2'|^{\delta}}{|x_2-y_2|^{m+\delta}}\frac{C_{\mathcal{B}}}{t^n}:$

\begin{align*}
|\langle \eta_t^{x_1},T^1\xi_t^{x_1}\rangle(x_2,y_2)-|\langle \eta_t^{x_1},T^1\xi_t^{x_1}\rangle(x_2',y_2)|&=|\langle \eta_t^{x_1},\left(K_2(x_2,y_2)-K_2(x_2',y_2)\right)\xi_t^{x_1}\rangle|\\
&\leq\Vert K_2(x_2,y_2)-K_2(x_2',y_2)\Vert_{2\rightarrow 2} \ \Vert\xi_t^{x_1}\Vert_2 \ \Vert \eta_t^{x_1}\Vert_2\\
&\leq C \frac{|x_2-x_2'|}{|x_2-y_2|^{m+\delta}} \cdot \frac{C_{\mathcal{B}}}{t^n}
\end{align*}

\item By a symmetric argument we have $|\langle \eta_t^{x_1}T^1\xi_t^{x_1}\rangle(x_2,y_2) - \langle \eta_t^{x_1}T^1\xi_t^{x_1}\rangle(x_2,y_2')|\leq C \frac{|y_2-y_2'|^{\delta}}{|x_2-y_2|^{m+\delta}}\frac{C_{\mathcal{B}}}{t^n}.$

\end{itemize}

Now we are left to prove that $\Vert K^1_{\xi_t^{x_1},\eta_t^{x_1}}\Vert_{2\rightarrow 2}\leq\frac{\tilde{C}_{\mathcal{B}}}{t^n}$.

Since we have proven that $K^1_{\xi_t^{x_1},\eta_t^{x_1}}$ has a $\delta$-standard kernel we are in the conditions of using Theorem \ref{1auxT1} to determine the $L^2$ boundedness bound, i.e., by duality we are reduced to prove that

\begin{align*}
&\vert\langle\langle \eta_t^{x_1},T^1\xi_t^{x_1}\rangle \chi_V, f_V\rangle\vert+\vert\langle f_V,\langle \eta_t^{x_1},T^1\xi_t^{x_1}\rangle \chi_V\rangle\vert=\\
&\vert\langle \eta_t^{x_1}\otimes \chi_V, T\xi_t^{x_1}\otimes f_V\rangle\vert+\vert\langle \eta_t^{x_1}\otimes f_V, T\xi_t^{x_1}\otimes \chi_V\rangle\vert\leq \frac{\tilde{C}_{\mathcal{B}}}{t^n}|V|
\end{align*}

\noindent for all cubes $V$ in $\mathbb{R}^m$ and all $f_V\in L^{\infty}(V)$ such that $\Vert f_V\Vert_{\infty}\leq 1$ which is satisfied by the $L^2$ to $L^2$ boundedness of the operator which ends our proof.
\end{proof}

\begin{remark}
We haven't included conditions \eqref{mixedwbp} -- \eqref{diagbmo4}  in the statement of the proof because they are a consequence of the $L^2$ to $L^2$ boundedness of the operator.
\end{remark}

On \cite{HyM2}, it was stated that if an operator $T$ defined as in \eqref{defmixoperator} satisfied conditions \eqref{size}--\eqref{adapted1}, \eqref{mixedwbp}--\eqref{diagbmo4} and $T1$, $T^*1$, $\tilde{T}1$ and $\tilde{T}^*1$ lie in BMO then the operator T could be extended to an $L^2$ to $L^2$ bounded operator.

It was also stated that if $T$ is a bi-parameter $\delta-SIO$ that can be extended to an $L^2$ to $L^2$ bounded operator then $T1$ and $T^*1$ lie in BMO. If in addition $\tilde{T}$ can be extended to an $L^2$ to $L^2$ bounded operator $\tilde{T}1$ and $\tilde{T}^*1$ lie in BMO also.

It was missing, and we have just proven, it's that if $T$ is an operator defined as in \eqref{defmixoperator} that satisfies conditions \eqref{size}--\eqref{adapted1} and can be extended to an $L^2$ to $L^2$ bounded operator then $T$ is a bi-parameter $\delta-SIO$.

As a consequence we can answer the following open question left in \cite{PV} and \cite{HyM2}.

\begin{corollary}
Let T be an operator defined as in \eqref{defmixoperator} that can be extended to an $L^2$ to $L^2$ bounded bi-parameter operator and whose kernel satisfies \eqref{size}--\eqref{adapted1}. Then $T1$ and $T^*1$ lie in BMO and it has the WBP in the classical sense.
\end{corollary}
 

\section{Classical conditions imply mixed type conditions \label{s5}}

We have proven that the mixed type conditions imply the classical conditions, so in this section we are going to proceed to prove the converse direction, i.e., that the classical conditions imply the mixed type conditions.

\begin{theorem}
Let $T$ be a bi-parameter $\delta$-SIO as defined in Definition \ref{deltasio} satisfying the WBP \eqref{classWBP}, then $T$ satisfies conditions \eqref{size}--\eqref{adapted1}. If in addition the operator is $L^2$ to $L^2$ bounded, then it also satisfies conditions \eqref{mixedwbp}--\eqref{diagbmo4}.
\end{theorem}

\begin{proof}
That the operator satisfies \eqref{size}--\eqref{adapted1} can be deduced directly from the definition of $\delta$-Calder\'on-Zygmund kernel. The size condition, H\"older conditions and mixed H\"older and size conditions (\eqref{size}--\eqref{mixed4}) are consequence of the pointwise conditions of the kernel while the separated H\"older and size conditions (\eqref{sep1}--\eqref{adapted1}) are consequence of the $L^2$ boundedness conditions of the kernels with constant $C(f_j,g_j)\leq C \Vert f_j\Vert_2\cdot \Vert g_j\Vert_2$ for $j=1,2$.

Finally, that the operator satisfies conditions \eqref{mixedwbp}--\eqref{diagbmo4} is a trivial consequence of the $L^2$ boundedness of the operator.
\end{proof}

\begin{remark}

It is worth noticing that  Pott and Villarroya original conditions differs slightly from the mixed type conditions that we have used in this paper. While we have used characteristic function and cube adapted functions in conditions \eqref{adapted1}--\eqref{diagbmo4}, they used instead some bump functions which has not necessarily compact support. 

That the above result can also be proven for the \cite{PV} conditions it is left for the reader. We would like to point out that in the uniparametric setting, we can indiscriminately test our operator on characteristic functions or in bump functions (c.f. \cite{G}). If we add that observation with the fact that we have used uniparametric results along the proofs of this paper, one can get an idea of the blueprint for proving such results for the \cite{PV} conditions.

\end{remark}

\end{document}